\documentclass[sn-mathphys,Numbered]{sn-jnl}

\usepackage{graphicx}%
\usepackage{multirow}%
\usepackage{amsmath,amssymb,amsfonts}%
\usepackage{amsthm}%
\usepackage{mathrsfs}%
\usepackage[title]{appendix}%
\usepackage{xcolor}%
\usepackage{textcomp}%
\usepackage{manyfoot}%
\usepackage{booktabs}%
\usepackage{algorithm}%
\usepackage{algorithmicx}%
\usepackage{algpseudocode}%
\usepackage{listings}
\usepackage{enumitem}%

\theoremstyle{thmstyleone}%
\newtheorem{thm}{Theorem}[section]
\newtheorem{cor}[thm]{Corollary}
\newtheorem{lem}[thm]{Lemma}

\theoremstyle{definition}

\theoremstyle{remark}
\newtheorem{rem}[thm]{Remark}
\newtheorem{ex}[thm]{Example}
\usepackage{caption}
\usepackage{subcaption}
\usepackage{tikz}
\tikzset{->-/.style={decoration={
			markings,
			mark=at position #1 with {\arrow{>}}},postaction={decorate}}}
\usetikzlibrary{shapes,backgrounds,calc}
\usepackage{graphics}
\theoremstyle{thmstyletwo}%

\theoremstyle{thmstylethree}%

\begin{document}

\title[Some central moments inequalities with applications]{Some central moments inequalities with applications}


\author[1]{\fnm{Mamta} \sur{Verma}}\email{mamtav.ma.22@nitj.ac.in; thakurrk@nitj.ac.in}
\equalcont{These authors contributed equally to this work.}

\author*[2]{\fnm{Ravinder} \sur{Kumar}}\email{thakurrk@nitj.ac.in}
\equalcont{These authors contributed equally to this work.}


\affil[1,2]{\orgdiv{Department of Mathematics and Computing}, \orgname{Dr BR Ambedkar National Institute of
		Technology Jalandhar}, \orgaddress{\street{GT Road}, \city{Jalandhar}, \postcode{144008}, \state{Punjab}, \country{India}}}

%


\abstract{In this paper, we first derive an inequality involving central moments for $n$ real numbers, which in turn provides an extension of Theorem 2.2 of Wolkowicz and Styan \cite{wol80}. Furthermore, we present refinements of various inequalities obtained by Sharma et al. \cite{sha12,raj15,raj18,raj20} involving central moments for the case of $n$ distinct integers. Moreover, we provide applications of our results in matrix theory and the theory of polynomial equations.}

\keywords{Central moments; eigenvalues; trace; spread; polynomial; roots; span}


\pacs[Mathematics Subject Classification]{Primary 60E15; Secondary 15A45}

\maketitle

\section{Introduction} \label{sec1}
Let 
$x_1 ,x_2 , \cdots , x_n$ be $n$ real numbers. Then their arithmetic mean is   \begin{equation*}\label{ie1}
	\bar{x}=\frac{1}{n}\sum_{i=1}^{n}x_{i},
\end{equation*}
and the $r$-th central moment of these numbers about origin and mean $\bar{x}$, respectively  be defined as   
\begin{equation*}\label{ie2}
	m_r'=\frac{1}{n}\sum_{i=1}^{n} {x_i^{r}}, \quad \text{and} \quad	m_r=\frac{1}{n}\sum_{i=1}^{n} {\left(x_i-\bar{x}\right)^{r}},
\end{equation*}where $r$ is any positive integer.\\Let $M=x_1 \geq x_2 \geq \cdots \geq x_n=m$ be $n$ real numbers and for $1\leq k \leq j \leq n$, denote 
\begin{equation}\label{ce2}
	x_{(k,j)}=\sum_{i=k}^{j}\frac{x_i}{j-k+1}.
\end{equation}
Then $	x_{(1,n)}=\bar{x},$
and $x_{(j,j)}=x_j.$\\
The lower bound for ${m_2}$ in terms of $x_j$ and $\bar{x}$ was given by the Samuelson inequality \cite{sam68}, that is,
\begin{equation}\label{iew1}
	m_2\geq \frac{1}{n-1}(x_{j}-\bar{x})^{2}, \quad 1\leq j \leq n,
\end{equation}
or equivalently,
\begin{equation*}
	\bar{x} - \sqrt{(n-1)m_2}\leq x_{j}\leq \bar{x} + \sqrt{(n-1)m_2}.
\end{equation*}
In \cite{wol80}, Wolkowicz and Styan proved that for $1 \leq k  \leq j \leq n $,
\begin{equation}\label{ieq4}
	\begin{array}{lcl}
		\bar{x}-\sqrt{\dfrac{k-1}{n-k+1}m_{2}}\leq
		x _{(k,j)}\leq \bar{x}+\sqrt{\dfrac{n-j}{j}m_{2}}, &  &  
	\end{array} 
\end{equation}
and therefore,
\begin{equation}\label{ine3}
	\begin{array}{lcl}
		
		\bar{x}-\sqrt{\dfrac{j-1}{n-j+1}m_{2}}\leq
		x_{j}\leq \bar{x}+\sqrt{\dfrac{n-j}{j}m_{2}}.
		
	\end{array}
\end{equation}
Nagy \cite{nagy18} proved the following lower bound for $m_2$:
\begin{equation}\label{i2}
	m_2\geq  \frac{(M-m)^2}{2n}.
\end{equation}Sharma et al. \cite{raj10,sha12,raj15,raj18,raj20} have extended these results to higher moments and provided more refined bounds. For instance, Sharma et al. \cite{sha12} obtained bounds for the third central moment in terms of the first two moments,
{ \begin{equation}\label{it1}
		\dfrac{m^2{\bar{x}}^2-m\left( m+\bar{x}\right)m_2'+{m_2'^2}}{\bar{x}-m}\leq {m_3}' \leq\dfrac{ M\left( M+\bar{x}\right)m_2'-M^2{\bar{x}}^2-{m_2'^2}}{M-\bar{x}}.
\end{equation}}Sharma et al. \cite{raj18} also derived bounds for the fouth central moment involving  first three central moments,
{\begin{equation}\label{it3}
		{m_4} \leq  \left( M-\bar{x}\right) \left( \bar{x}-m\right)m_2+(M+m-2\bar{x}){m_3}-\dfrac{ \left( m_3-(m+M-2\bar{x})m_2\right)^2}{\left( M-\bar{x}\right) \left( \bar{x}-m\right)-m_2}
\end{equation}}and 
\begin{equation}\label{it4}
	m_4-m_2^2-\frac{m_3^2}{m_2} \leq \frac{(M-m)^4}{64}.
\end{equation}
They also proved the extension of \eqref{i2},
\begin{equation}\label{i3}
	m_{2r}\geq  \frac{(M-m)^{2r}}{2^{2r-1}n}.
\end{equation} 
A generalization of inequality \eqref{iew1} was obtained by Sharma and Saini \cite{raj15},
\begin{equation}\label{m2}
	m_{2r}\geq 	\frac{1+(n-1)^{2r-1}}{n(n-1)^{2r-1}} \left(x_j-\bar{x}\right) ^{2r}, \quad 1\leq j \leq n.
\end{equation}The manuscript is structured as follows: Section \ref{sec1} serves as an introduction. In Section \ref{sec2}, we derive extensions of \eqref{ieq4} and \eqref{ine3}, presented in Theorem \ref{T1}. In our recent work \cite{ver25}, we established refined inequalities for $m_2$ in the case of $n$ distinct integers, leading to improvements upon classical inequalities such as the Bhatia-Davis inequality, Popoviciu inequality, and Nagy inequality. Motivated by this, we now present refinements for higher-order moments, particularly the third and fourth moments. Specifically, we provide refinements of \eqref{it1}, \eqref{it4}, and \eqref{i3} in Theorems \ref{T2}, \ref{T3}, and \ref{T5}, respectively, for the case of  $n$ distinct integers. Our results also have significant applications in matrix theory and the theory of polynomial equations, which are explored in Sections \ref{sec3} and \ref{sec4}. Furthermore, we include illustrative examples (Examples \ref{se1}, \ref{x4}, \ref{x8}, and \ref{x3}) and provide comparisons with existing bounds in the literature to demonstrate the effectiveness of our results.

\section{Main results}\label{sec2}
We start our work by presenting the following theorem which provides a generalization  of Theorem  2.2 of Wolkowicz and Stayn \cite{wol80}.
\begin{thm}\label{T1}
	
	Let $ m_{2r} $ be the $2r$-th central moment of $n$ real numbers $x_1 \geq x_2 \geq \cdots \geq x_n$, and let $	x_{(k,j)}$ be defined as in \eqref{ce2}. Then for any positive integer $r$,
	{\scriptsize \begin{align}\label{2e1} 
			\bar{x}-\left(\frac{n(k-1)^{2r-1}m_{2r}}{(n-k+1)^{2r}+(n-k+1)(k-1)^{2r-1}}\right)      ^{\frac{1}{2r}}\leq		x_{(k,j)} \leq \bar{x}+\left(  \frac{n(n-j)^{2r-1}m_{2r}}{j^{2r}+j(n-j)^{2r-1}} \right)     ^{\frac{1}{2r}}.\end{align}}
	If $(k,j)=(1,n),$ the inequality string fails. If $(k,j) \ne (1,n),$ equality exists on the right side of \eqref{2e1} if and only if 
	\begin{align}\label{2e2}
		x_1=x_2=\cdots =x_j  ~\text{and}~ x_{j+1}=x_{j+2}=\cdots=x_n,
	\end{align}
	and on the left if and only if 
	\begin{align*}\label{eqr1}
		x_1=x_2=\cdots =x_{k-1} ~\text{and}~x_{k}=x_{k+1}=\cdots=x_n.
	\end{align*}
	Furthermore, for $1\leq j \leq n$,
	{\scriptsize \begin{equation}\label{2e3} 
			\bar{x}-\left(\frac{n(j-1)^{2r-1}m_{2r}}{(n-j+1)^{2r}+(n-j+1)(   j-1)^{2r-1}}\right)      ^{\frac{1}{2r}}\leq		x_{j} \leq \bar{x}+\left(  \frac{n(n-j)^{2r-1}m_{2r}}{j^{2r}+j(n-j)^{2r-1}} \right)     ^{\frac{1}{2r}}.
	\end{equation}}
	Equality exists on the right side of \eqref{2e3} if and only if \eqref{2e2} exists,
	and on the left if and only if 
	$$x_1=x_2=\cdots =x_{j-1} ~ \text{and}~x_{j}=x_{j+1}=\cdots=x_n.
	$$
\end{thm}
\begin{proof} 
	We first prove the right-hand side inequality of \eqref{2e1}. When $j=n,$ inequality is trivially true. Now, consider the case when $1\leq j \leq n-1.$ We write
	\begin{equation}\label{2e4}
		m_{2r}=\frac{j}{n} \left(\frac{1}{j}\sum_{i=1}^{j}\left(x_i-\bar{x}\right) ^{2r}\right)+\frac{n-j}{n} \left(\frac{1}{n-j}\sum_{i=j+1}^{n}\left(x_i-\bar{x}\right) ^{2r}\right),
	\end{equation}where $1\leq j \leq n-1.$
	Also, for $n$ positive real numbers $d_1,d_2,\dots,d_n,$ we have
	\begin{equation*}\label{2e5}
		\frac{1}{n}\sum_{i=1}^{n}{d_i}^{p} \geq \left( \frac{1}{n}\sum_{i=1}^{n}{d_i}\right)^{p},      
	\end{equation*}
	where $p$ is any positive integer. Therefore, \eqref{2e4} gives
	\begin{equation}\label{e22}
		m_{2r}\geq\frac{j}{n} \left(\frac{1}{j}\sum_{i=1}^{j}\left(x_i-\bar{x}\right) ^{2}\right)^r+\frac{n-j}{n} \left(\frac{1}{n-j}\sum_{i=j+1}^{n}\left(x_i-\bar{x}\right) ^{2}\right)^r.
	\end{equation}
	The Cauchy-Schwarz inequality for real numbers $s_i $ and $r_i  ;~i=1,2,\dots,n$ is given by
	\begin{equation}\label{e4.1}
		\sum_{i=1}^{n} {s_i}^2\sum_{i=1}^{n} {r_i}^2 \geq \left(\sum_{i=1}^{n} {s_i}r_i \right)^2.
	\end{equation}
	Applying \eqref{e4.1} for $r_i=\frac{1}{j}$ and $s_i=\left(x_i-\bar{x}\right); i=1,\dots,j$, we obtain that
	\begin{equation}\label{e5}
		\frac{1}{j}\sum_{i=1}^{j}\left(x_i-\bar{x}\right)^{2} \geq \left(\frac{1}{j}\sum_{i=1}^{j}\left(x_i-\bar{x}\right) \right)^2.
	\end{equation}
	Similarly, we have
	\begin{equation}\label{e51}
		\frac{1}{n-j}\sum_{i=j+1}^{n}\left(x_i-\bar{x}\right)^{2} \geq \left(\frac{1}{n-j}\sum_{i=j+1}^{n}\left(x_i-\bar{x}\right) \right)^2.
	\end{equation}
	Also, the first moment about the mean is zero. Therefore,
	$$\sum_{i=1}^{n}\left(x_i-\bar{x}\right)=0,$$
	and hence,
	\begin{equation}\label{e6}
		\sum_{i=j+1,}^{n}\left(x_i-\bar{x}\right)=-\sum_{i=1,}^{j}\left(x_i-\bar{x}\right).
	\end{equation}Combining \eqref{e22}, \eqref{e5}, \eqref{e51} and \eqref{e6}, we have
	\begin{equation}\label{e8}
		m_{2r}\geq \frac{j^{2r}+j(n-j)^{2r-1}}{n(n-j)^{2r-1}} \left(x_{(1,j)}-\bar{x}\right) ^{2r}.
	\end{equation}
	It is easy to see that $x_{(k,j)} \leq x_{(1,j)}$. The right-hand side inequality of \eqref{2e1} follows immediately from \eqref{e8} using the fact that $x^q \in  [a, b]$ implies $x \in	[a^{1/q}, b^{1/q}] $.\\
	To prove the left-hand side inequality of \eqref{2e1}, let $ q=n-k+1 $ and $w_q=-x_k$. Then $w_1 \geq w_2 \geq \cdots \geq w_k\geq \cdots \geq w_n$. Again, by using the analogous arguments for $w_k$ as used above, we obtain that 
	\begin{equation}\label{e9}
		m_{2r}\geq 	\frac{(n-k+1)^{2r}+(n-k+1)(k-1)^{2r-1}}{n(k-1)^{2r-1}} \left(\bar{x}-x_{(k,n)}\right) ^{2r}.
	\end{equation}
	A little calculation in \eqref{e9} directly leads to the left-hand side inequality of \eqref{2e1} by using the fact that $x_{(k,j)}\geq x_{(k,n)}$. The inequality \eqref{2e3} follows immediately from \eqref{2e1} by setting $k=j$. 
\end{proof} 
\begin{rem}
	\begin{itemize}
		\item[(i)] For $r=1$, Theorem \ref{T1} reduces to  \eqref{ieq4}. Hence, Theorem \ref{T1} provides an extension of \eqref{ieq4}. Also, when $r=2,$ Theorem \ref{T1} yields a better lower bound for $x_2$ and a better upper bound for $x_{n-1}$ than \eqref{ine3}, by using the fact that $\beta_2=\dfrac{m_4}{{m_2}^2}\leq \dfrac{n^2-3n+3}{n-1}$; see \cite{dal87}.
		\item[(ii)]
		Under the hypothesis of Theorem \ref{T1}, we have			\begin{equation}\label{re1}
			m_{2r}\geq 
			\begin{cases}
				\frac{j^{2r}+j(n-j)^{2r-1}}{n(n-j)^{2r-1}} \left(x_j-\bar{x}\right) ^{2r}, &\text{if} ~~x_j\geq \bar{x},\\
				\frac{(n-j+1)^{2r}+(n-j+1)(j-1)^{2r-1}}{n(j-1)^{2r-1}} \left(x_j-\bar{x}\right) ^{2r}, & \text{if}~~ x_j\leq\bar{x}.
			\end{cases}
		\end{equation}
		Note that for $j=1$ and $j=n$, \eqref{m2} and \eqref{re1} are identical. For $2\leq j \leq n-1$, \eqref{re1} may provides a better estimate than \eqref{m2}. For example,
		let $T=\{10,9,8,2,1\}$. Then, from \eqref{m2}, we have $m_4 \geq 126.9531$, while our bound \eqref{re1} gives a better estimate: $ m_4 \geq 132.7407$.
\end{itemize}
\end{rem} 
\begin{cor}
If $ d_j $ is the $j$-th absolute deviation from the mean $(\bar{x})$ of n real numbers $x_1 \geq x_2 \geq \cdots \geq x_j \geq \cdots \geq x_n, $ then for $1\leq j \leq n$,
\scriptsize{\begin{align*}\label {1c1} 
		d_j \leq max \left\lbrace \left(  \frac{n(n-j)^{2r-1}m_{2r}}{j^{2r}+j(n-j)^{2r-1}} \right)     ^{\frac{1}{2r}},\left(  \frac{n(j-1)^{2r-1}m_{2r}}{(n-j+1)^{2r}+(n-j+1)(j-1)^{2r-1}} \right)     ^{\frac{1}{2r}}\right\rbrace. \end{align*}}

\end{cor}
\begin{proof}
The proof is an immediate consequence of Theorem \ref{T1}.
\end{proof}
We now introduce the following lemmas, which play a pivotal role in the proof of Theorem \ref{T2}.
\begin{lem}\label{L1}
Let $p(x)=x^2+bx-c, (b,c\geq 0)$ be a polynomial equation with real roots. Then $t$ is the greatest integer of the largest root of $p(x),$ if it satisfies the following conditions:
\begin{enumerate}[label=(\alph*)]
	\item $c-t^2-tb\geq 0,$ 
	\item $(t+1)^2+tb+b-c\geq 0 $.
\end{enumerate}
\end{lem}
\begin{proof}
Let $k=\frac{-b+\sqrt{b^2+4c}}{2} \geq 0$ be the largest root of $p(x)$. Then $t$ is the greatest integer of $k$, if $k-t\geq 0$ and $k-t \leq 1$. 
The condition $(a)$ follows by $k-t \geq 0,$ that is, $\sqrt{b^2+4c}\geq 2t+b$ because  $2t+b\geq 0$ as $k\geq 0 $ and $b\geq 0.$ Likewise, we obtain the condition $(b)$ on using $k-t \leq 1$.		  
\end{proof}
\begin{lem}\label{L2}
Let $p(x)=-x^2+bx-c, (b,c\geq 0)$ be a polynomial with real roots. Then $t$ is greatest integer of the smallest root of $p(x),$ if it satisfies the following conditions:
\begin{enumerate}[label=(\alph*)]
	\item $c+t^2-tb\geq 0,$ 
	\item $tb-(t+1)^2+b-c\geq 0 $.
\end{enumerate}
\end{lem}
\begin{proof} The proof of Lemma \ref{L2} follows by using similar arguments as in the proof of Lemma \ref{L1}. 
\end{proof}
We now present a theorem that provides inequalities for $m_3'$, improving the Sharma et al. inequalities \eqref{it1} for the case of $n$ distinct integers.
\begin{thm}\label{T2}
Let $n\geq 3$ and $ m=x_1 < x_2 < \cdots < x_n=M$ be any $n$ distinct integers. Then
{ \begin{equation}\label{2t1}
		\dfrac{m^2{\bar{x}}^2-m\left( m+\bar{x}\right)m_2'+{m_2'^2}}{\bar{x}-m}+\beta_1\leq {m_3}' \leq\dfrac{ M\left( M+\bar{x}\right)m_2'-M^2{\bar{x}}^2-{m_2'^2}}{M-\bar{x}}-{\beta_1}, 
\end{equation}}with
{ \begin{equation}\label{2t2}
		\beta_1 =\begin{cases}
			\frac{1}{324}(n-3)(9n^2-23n+12), &\text{if} ~~n=3q,\\
			\frac{1}{324n}(n-1)(9n^3-41n+46n+4), &\text{if} ~~n=3q+1,\\
			\frac{1}{324n}(n-2)(9n^3-32n^2+47n-20), &\text{if} ~~n=3q+2,\\
		\end{cases}
	\end{equation}where $q$ is a positive integer. }
	\end{thm}
	\begin{proof}
Let $\alpha\in \left[ x_k,x_{k+1}\right]$  with $1 \leq k \leq n-1$ be any real number. Then
\begin{equation}\label{2t3}
	\hspace{-1.8cm}	x_i \leq \alpha-(k-i), \quad  i=1,2,\dots,k,
	\vspace{-.2cm}
\end{equation}\text{and}
\begin{equation}\label{2t4}
	x_i \geq \alpha+(i-k-1),\quad i=k+1,k+2,\dots,n.
\end{equation}  Note that the function $\left( x-\alpha\right) ^2 $ is monotonically decreasing in  $\left[ m,\alpha\right]$ and monotonically increasing in $ \left[\alpha,M\right]$. Therefore, from \eqref{2t3} and \eqref{2t4}, we find that
{ \begin{equation}\label{2t5}
		\left( x_i-\alpha\right) ^2 \geq \begin{cases}
			(k-i)^2, &\text{when} ~~i=1,2,\dots,k,\\
			(i-k-1)^2, &\text{when} ~~i=k+1,k+2,\dots,n.
		\end{cases}
\end{equation}}Also, we have	
\begin{equation}\label{2t6}
	\hspace{-1.8cm}	\left( M-x_i\right) \geq (n-i), \quad  i=1,2,\dots,n
\end{equation}
and 
\begin{equation}\label{2t7}
	\hspace{-1.8cm}	\left( x_i-m\right) \geq (i-1), \quad  i=1,2,\dots,n.
\end{equation}
We now prove the right-hand side inequality of \eqref{2t1}. From \eqref{2t5} and \eqref{2t6}, we obtain the following inequality: {\small	\begin{equation*}
		\begin{aligned}\label{2t8}
			\sum_{i=1}^{n}\left( x_i-\alpha\right) ^2(M-x_i)
			&\geq
			\sum_{i=1}^{k} (k-i)^2(n-i)+\sum_{i=k+1}^{n}(i-k-1)^2(n-i)=\gamma_1~\text(say),\\
		\end{aligned}
\end{equation*}}which implies that 
{	\begin{equation*}\label{2t9}
		-\sum_{i=1}^{n}{x_i}^3	+\sum_{i=1}^{n}{x_i}^2\left(M+2\alpha\right)	-\sum_{i=1}^{n}{x_i}(2\alpha M+\alpha^2)+n\alpha^2M\geq \gamma_1
\end{equation*}}and therefore, 	{	\begin{equation}\label{2t10}
		m_3'	\leq m_2'\left(M+2\alpha\right)	-\bar{x}(2\alpha M+\alpha^2)+\alpha^2M-\frac{ \gamma_1}{n}.
\end{equation}}The inequality \eqref{2t10} holds for all real number $\alpha$ and provides the least upper bound for 
{	\begin{equation}\label{2t11}
		\alpha=\frac{M\bar{x}-m_2'}{M-\bar{x}} \in [m,M].
\end{equation}}Substituting the value of $\alpha $ from \eqref{2t11} in \eqref{2t10}, we get
\begin{equation}\label{23t11}
	{m_3}' \leq\dfrac{ M\left( M+\bar{x}\right)m_2'-M^2{\bar{x}}^2-{m_2'^2}}{M-\bar{x}}-\frac{ \gamma_1}{n}.
\end{equation}
Also, we have
\begin{equation}
	\begin{aligned}\label{23t5}
		\gamma_1&=
		\sum_{i=1}^{k} (k-i)^2(n-i)+\sum_{i=k+1}^{n}(i-k-1)^2(n-i) \\
		&=(n-k)\sum_{i=1}^{k-1} i^2+\sum_{i=1}^{k-1} i^3+(n-k-1)\sum_{i=1}^{n-k-1}i^2-\sum_{i=1}^{n-k-1}i^3\\&= \frac{k^3}{3}+\frac{(n^2-3n+1)k^2
		}{2}-\frac{(2n^3-6n^2+4n-1)k}{6}+\frac{(n^4-4n^3+5n^2-2n)}{12} .
	\end{aligned}
\end{equation}
Now, to find the minimum value of $\gamma_1$, we write  $$\gamma_1=\frac{k^3}{3}+\frac{bk^2}{2}-ck+d=f(k),$$
where $b= (n^2-3n+1)\geq 0,c=\frac{(2n^3-6n^2+4n+1)}{6} \geq 0, ~~\text{and} ~d=-\frac{(n^4-4n^3+5n^2-2n)}{12}.$ 
It is easy to see that $f(k)$ achieves its minimum value at $k=k_1,$ where
\begin{equation*}
	k_1=	\dfrac{-b+\sqrt{{b^2+4c}}}{2} \in [1,n].
\end{equation*}
Note that $f(k)$ is monotonically decreasing in  $[ 1,k_1]$ and monotonically increasing in $[k_1,n].$ However, since $k$ is integer, therefore $f(k)$ achieves its minimum value at either $k=[k_1]$ or $x= [k_1]+1$, where $[k_1]$ denotes the greatest integer of $k_1$. Claim that
{ \begin{equation*}\label{2t13}
		[k_1]=\begin{cases}
			q-1, &\text{when} ~~~~n=3q ,\\
			
			q, &\text{when}~~~~n=3q+1 ~~~~\text{or}~~~~ n=3q+2,
		\end{cases}
\end{equation*}}where $q$ is a positive integer.\\
When $n=3q$. The claim follows by using Lemma \ref{L1} for $f'(k)=k^2+bk-c$  and $t=q-1$. Note that $f(q)-f(q-1)=-\frac{7q^2}{2}+\frac{5q}{2}\leq 0$  for $q \geq 1.$ Hence, $f(k)$ achieves its minimum value at $k=q$.  Substituting $k=q=\frac{n}{3}$ in \eqref{23t5}, we get $\gamma_1 \geq 	\frac{1}{324}n(n-3)(9n^2-23n+12).$ \\Likewise, when $n=3q+1$ or $n=3q+2$, $f(k)$ attains its minimum value at $k=q$ or $k=q+1$, respectively. The right-hand side inequality of \eqref{2t1} follows by combining \eqref{23t11} with the minimum value of $\gamma_1.$\\
Similarly, we prove  the left-hand side inequality of \eqref{2t1}. Applying \eqref{2t5} and \eqref{2t7}, we obtain
{	\begin{equation*}
		\begin{aligned}\label{2t14}
			\sum_{i=1}^{n}\left( x_i-\alpha\right) ^2(x_i-m)
			\geq \sum_{i=1}^{k} (k-i)^2(i-1)+\sum_{i=k+1}^{n}(i-k-1)^2(i-1).
		\end{aligned}
\end{equation*}}and therefore, {\begin{equation*}\label{2t15}
		m_3'	\geq m_2'\left(m+2\alpha\right)	-\bar{x}(2\alpha m+\alpha^2)+\alpha^2m+\frac{ \gamma_2}{n},
\end{equation*}}where
\begin{equation*}\label{2t16}
	\begin{aligned}
		\gamma_2&=
		\sum_{i=1}^{k} (k-i)^2(i-1)+\sum_{i=k+1}^{n}(i-k-1)^2(i-1) \\
		&=(k-1)\sum_{i=1}^{k-1} i^2-\sum_{i=1}^{k-1} i^3+k\sum_{i=1}^{n-k-1}i^2+\sum_{i=1}^{n-k-1}i^3\\&= \frac{-k^3}{3}+\frac{(n^2-n+1)k^2
		}{2}-\frac{(4n^3-6n^2+2n+1)k}{6}+\frac{(n^4-2n^3+n^2)}{4} .
	\end{aligned}
\end{equation*}
It is easy to see that $\gamma_2$ attains its minimum value at  \begin{equation*}
	k=	\dfrac{(n^2-n+1)-\sqrt{\frac{3n^4-14n^3+21n^2-10n+1}{3}}}{2} \in [1,n].
\end{equation*}
The left-hand side inequality of \eqref{2t1} now follows from Lemma \ref{L2},  along with analogous arguments as used above, and  
{ \begin{equation*}\label{2t17}
		[k]=\begin{cases}
			2q, &\text{when} ~~~~n=3q ~~~~\text{or}~~~~ n=3q+1,\\
			
			2q+1, &\text{when}\text{or}~~~~ n=3q+2,
		\end{cases}
\end{equation*}}where $q$ is a positive integer.
\end{proof} 
\begin{cor}\label{C1}
Let $n\geq 3$ and $ m=x_1 < x_2 < \cdots < x_n=M$ be any $n$ distinct integers. Then 
{\begin{equation}\label{c1}
		\dfrac{{m_2}^2-({\bar{x}-m})^2m_2}{\bar{x}-m} +\beta_1\leq {m_3} \leq\dfrac{ \left( M-\bar{x}\right)^2m_2-{m_2^2}}{M-\bar{x}}-\beta_1, 
\end{equation}}where $\beta_1$ is as defined in \eqref{2t2}. 
\end{cor}
\begin{proof}
The proof of the corollary follows directly by using the relations $m_2'=m_2+\bar{x}^2,$  and $m_3'=m_3+3m_2\bar{x}+\bar{x}^3.$ 
\end{proof} \noindent
In \cite{raj10,raj18,raj20}, Sharma et al. established the following inequalities:
{ \begin{equation}\label{2.25raj20}
	\dfrac{-(\bar{x}-m)^3}{4} \leq {m_3} \leq\dfrac{(M-\bar{x})^3}{4},
	\end{equation}} 
	\begin{equation}\label{2.2raj10}
m_2+\left(\frac{m_3}{2m_2}\right)^2\leq  \frac{(M-m)^2}{4}
\end{equation}and
\begin{equation}\label{2.6raj18}
|{m_3}|\leq  \frac{(M-m)^3}{6\sqrt{3}}.
\end{equation}
We now present some corollaries that provide improvements of \eqref{2.25raj20}, \eqref{2.2raj10} and \eqref{2.6raj18} for the case of $n$ distinct integers.
\begin{cor}\label{C2}
Let the hypothesis of Theorem \ref{T2}. Then
{ \begin{equation}\label{c3}
		\dfrac{-(\bar{x}-m)^3}{4} +\beta_1\leq {m_3} \leq\dfrac{(M-\bar{x})^3}{4}-\beta_1.
\end{equation}}
\end{cor}
\begin{proof}
The inequalities in \eqref{c3} are derived from \eqref{c1} by observing that the maximum and minimum values of the right-hand side and left-hand side expressions of \eqref{c1} are attained at $m_2=\dfrac{(M-\bar{x})^2}{2}$ and $m_2=\dfrac{(\bar{x}-m)^2}{2}$, respectively.
\end{proof}
\begin{cor}\label{C3}
Under the hypothesis of Theorem \ref{T2}, we have
\begin{equation}\label{3c1}
	m_2+\left(\frac{m_3+\beta_1}{2m_2}\right)^2\leq  \frac{(M-m)^2}{4}.
\end{equation} 

\end{cor}
\begin{proof}
From the right-hand side inequality of \eqref{c1}, we get
{ \begin{equation*}\label{3c2}
		m_2	(M-\bar{x})^2-(\beta_1+m_3)(M-\bar{x})-m_2^2 \geq 0,
\end{equation*}}which is a quadratic inequality in $M-\bar{x}$. Therefore, 
\begin{equation}\label{3c3}
	M\geq \bar{x}	+\frac{\beta_1+m_3+\sqrt{(\beta_1+m_3)^2+4m_2^3}}{2m_2}.
\end{equation}
Likewise, from the left-hand side inequality of \eqref{c1}, we find
\begin{equation}\label{3c4}
	m\leq \bar{x}	+\frac{\beta_1+m_3-\sqrt{(\beta_1+m_3)^2+4m_2^3}}{2m_2}.
\end{equation} The inequality \eqref{3c1} now follows by combining \eqref{3c3} and \eqref{3c4}. 
\end{proof}
\begin{cor}\label{C4}
Under the hypothesis of Theorem \ref{T2}, we have
\begin{equation}\label{4c1}
	{|m_3+\beta_1|}\leq  \frac{(M-m)^3}{6\sqrt{3}}.
\end{equation}

\end{cor}
\begin{proof}
From \eqref{3c1}, we find that 
\begin{equation}\label{4c2}
	(\beta_1+m_3)^2 \leq (M-m)^2m_2-4m_2^3.
\end{equation}
The inequality \eqref{4c1} follows directly from \eqref{4c2} by noting that the right-hand side expression of \eqref{4c2} attains its maximum value when 
$m_2=\frac{(M-m)^2}{{6}}$.
\end{proof}
Next, we present an upper bounds for $m_4$, which improves upon \eqref{it3} for the case of $n$ distinct integers.
\begin{thm}\label{T3}
Let $n\geq 3$ and $ m=x_1 < x_2 < \cdots < x_n=M$ be any $n$ distinct integers. Then 
{\small\begin{equation}\label{3t11}
		{m_4} \leq  \left( M-\bar{x}\right) \left( \bar{x}-m\right)m_2+(M+m-2\bar{x}){m_3}-\dfrac{ \left( m_3-(m+M-2\bar{x})m_2\right)^2}{\left( M-\bar{x}\right) \left( \bar{x}-m\right)-m_2}-\beta_2, 
\end{equation}} where 
{ \begin{equation*}\label{3t2}
		\beta_2 =\begin{cases}
			\frac{1}{480}(n-2)(n-4)(4n^2-11n+2), &\text{if} ~~~~n \text{ is even},\\
			
			\frac{1}{480n}(n-1)(n-3)(4n^3-19n+32n-5), &\text{if} ~~~~n \text{ is odd}.
		\end{cases}
\end{equation*}}
\end{thm}
\begin{proof}
By \eqref{2t5},~~ \eqref{2t6} \text{and} ~~\eqref{2t7}, we can write {\scriptsize	\begin{equation*}
		\begin{aligned}\label{3t3}
			&	\sum_{i=1}^{n}\left( x_i-\alpha\right) ^2(M-x_i)(x_i-m)\geq
			\sum_{i=1}^{k} (k-i)^2(n-i)(i-1)+\sum_{i=k+1}^{n}(i-k-1)^2(n-i)(i-1).
		\end{aligned}
\end{equation*}}
Consequently, 	{	\begin{align}\label{3t4}
		m_4'	&\leq\left(-mM +(M+m)\bar{x}-m_2'\right)\alpha^2	-2\left((m+M)m_2'-m_3'-Mm\bar{x}\right)\alpha \nonumber \\&\quad+(m+M)m_3'-Mmm_2'+\frac{ \gamma_3}{n},
\end{align}}where
{
	\begin{align*}\label{3t5}
		\gamma_3&= 
		\sum_{i=1}^{k} (k-i)^2(n-i)(i-1)+\sum_{i=k+1}^{n}(i-k-1)^2(n-i)(i-1) \nonumber \\ 
		&=(2k-n+1)\sum_{i=1}^{k-1} i^3-(k - n)(k - 1)\sum_{i=1}^{k-1} i^2-\sum_{i=1}^{k-1}i^4-\sum_{i=1}^{n-k-1}i^4\nonumber\\&\quad +(n-2k-1)\sum_{i=1}^{n-k-1}i^3-k(k-n+1)\sum_{i=1}^{n-k-1}i^2 \nonumber\\&=\frac{1}{6} \left( {k^4}-{2nk^3
		}+{(n^3-3n^2+5n-1)k^2}-{(n^4-4n^3+5n^2-n)k}\right)  \nonumber\\&\quad+\frac{3n^5-15n^4+25n^3-15n^2+2n}{60}.
\end{align*}}
Note that $\gamma_3$ attains its maximum value at $k=\frac{n}{2}$, when $n$ is even  or $k=\frac{n\pm1}{2}$, when $n$ is odd. Moreover, the right-hand side expression of \eqref{3t4} attains its maximum value at  $\alpha=\frac{(m+M)m_2'-m_3'-Mm\bar{x}}{-mM +(M+m)\bar{x}-m_2'}, $ and therefore
{\begin{equation}\label{3t19}
		{m_4'} \leq  -mMm_2'+(M+m){m_3'}-\dfrac{ \left( (m+M)m_2'-m_3'-Mm\bar{x}\right)^2}{-mM +(M+m)\bar{x}-m_2'}-\beta_2.
\end{equation}}The inequality \eqref{3t11}  now follows from \eqref{3t19} by using the relations $m_2'=m_2+\bar{x}^2,$ $m_3'=m_3+3m_2\bar{x}+\bar{x}^3$ and $m_4'=m_4+4m_3\bar{x}+6m_2\bar{x}^2+\bar{x}^4.$  
\end{proof}
\begin{cor} \label{C6} Under the hypothesis of  Theorem \ref{T3}, we have
\begin{equation}\label{5c1}
	m_4-m_2^2-\frac{m_3^2}{m_2} \leq \frac{(M-m)^4}{64}-\beta_2.
\end{equation}
\end{cor}	
\begin{proof}
From \eqref{3t11}, we have 
{\small	\begin{equation}\label{5c2}
		\begin{aligned}
			m_4m_2-m_2^3-m_3^2& \leq   \left( M-\bar{x}\right) \left( \bar{x}-m\right)m_2^2+(M+m-2\bar{x}){m_3}m_2\\&-\dfrac{ m_2\left( m_3-(m+M-2\bar{x})m_2\right)^2}{\left( M-\bar{x}\right) \left( \bar{x}-m\right)-m_2}-\beta_2m_2-m_2^3-m_3^2.
		\end{aligned}
\end{equation}}The right-hand side expression of \eqref{5c2} attains its maximum value at $$m_3=\frac{(M+m-2\bar{x})m_2}{2\left( M-\bar{x}\right)\left( \bar{x}-m\right)}\left(m_2+\left( M-\bar{x}\right) \left( \bar{x}-m\right)\right),$$ 
and therefore, 
{	\begin{equation*}\label{5c3}
		m_4m_2-m_2^3-m_3^2 \leq \left(1-\dfrac{m_2 }{\left( M-\bar{x}\right)\left( \bar{x}-m\right)} \right)\frac{(M-m)^2}{4}m_2^2-\beta_2m_2,
\end{equation*}}which implies that
{\	\begin{equation}\label{5c4}
		m_4-m_2^2-\frac{m_3^2}{m_2} \leq \left(1-\dfrac{m_2 }{\left( M-\bar{x}\right)\left( \bar{x}-m\right)} \right)\frac{(M-m)^2}{4}m_2-\beta_2.
\end{equation}}It is easy to see that the right-hand side expression of \eqref{5c4} achieves its maximum value at  $m_2=\frac{1}{2}\left( M-\bar{x}\right)\left( \bar{x}-m\right),$ and therefore \eqref{5c4} gives 
\begin{equation}\label{5c5}
	m_4-m_2^2-\frac{m_3^2}{m_2} \leq (M-\bar{x})(\bar{x}-m)\frac{(M-m)^2}{16}-\beta_2.
\end{equation}
Applying the A.M.-G.M. inequality on right-hand side of \eqref{5c5}, we get \eqref{5c1}.
\end{proof}Now, we derive an upper bound for $m_4$ that depends only on the number of integers. 
\begin{thm}\label{T4}
Let $m=x_1 < x_2 < \cdots < x_n=M$ be any $n$ distinct integers. Then, for $n \geq3$ 
{ \begin{equation}\label{4t1}
		m_4 \geq \begin{cases}
			\frac{(n-1)(n-2)(3n^2-6n-4)}{240},  &\text{if} ~~n \text{ is even},\\
			\frac{(n-1)(3n^4-12n^3+38n^2-52n+15)}{240n}, & \text{if}~~ n \text{ is odd}.
		\end{cases}
\end{equation}}
Furthermore, if $\bar{x}$ is an integer, then
{ \begin{equation}\label{4t2}
		\hspace{-1.7cm}m_4 \geq \begin{cases}
			\frac{3n^4+20n^2-8}{240},  &\text{if} ~~n \text{ is even},\\
			\frac{(n+1)(n-1)(3n^2-7)}{240}, & \text{if}~~ n \text{ is odd}.
		\end{cases}
\end{equation}}
The equality in Theorem \ref{T4} holds precisely, when  $ x_1,x_2 , \cdots , x_{n-1},x_n$ are $n$ consecutive integers assuming that $n$ is odd.
\end{thm}
\begin{proof}
Let $\bar{x}\in \left[ x_k,x_{k+1}\right] $  with $1 \leq k \leq n-1.$ Then, it is easy to see that
\begin{equation}\label{4t3}
	\hspace{-1.8cm}	\left( x_i-\bar{x}\right) ^4 \geq (k-i)^4, \quad  i=1,2,\dots,k,
	\vspace{-.2cm}
\end{equation}\text{and}
\begin{equation}\label{4t4}
	\left( x_i-\bar{x}\right) ^4 \geq (i-k-1)^4,\quad i=k+1,k+2,\dots,n.
\end{equation}
Combining \eqref{4t3} and \eqref{4t4}, we obtain that
{ \begin{align}\label{4t5}
		\sum_{i=1}^{n}\left( x_i-\bar{x}\right) ^4&
		\geq 
		\sum_{i=1}^{k} (k-i)^4+\sum_{i=k+1}^{n}(i-k-1)^4 \nonumber\\&=\sum_{i=1}^{k-1} i^4+\sum_{i=1}^{n-k-1}i^4 \nonumber\\&=(n-1)\left( k^4-2k^3n+n(n-1)k^2-n^2(n-1)k\right)  \nonumber\\& \quad+ \frac{n(n-1)}{30} \left( 6n^3-9n^2+n+1\right) .
\end{align}}Note that the right-hand side expression of \eqref{4t5} is monotonically decreasing on the interval $[1,\frac{n}{2}]$, and monotonically increasing on $[\frac{n}{2},n].$ However, since \( k \) is an integer, the minimum value of the expression is attained at  $k=\frac{n}{2}$, when $n$ is even  or $k=\frac{n\pm1}{2}$, when $n$ is odd. Hence, \eqref{4t1} follows.\\
Next, suppose that $\bar{x}$ is an integer, without loss of generality, take either $\bar{x}=x_k$ or $\bar{x} \in (x_k,x_{k+1})$, then
\begin{equation}\label{4t6}
	\hspace{-1.2cm}	\left( x_i-\bar{x}\right) ^4\geq (k-i)^4, \quad i=1,2,\dots,n.
\end{equation}
Using a similar procedure as in the derivation of \eqref{4t1}, the inequality \eqref{4t2} now follows from \eqref{4t6}.
\end{proof}We now present an improvement of \eqref{i3} for $r=2$ in the case of $n$ distinct integers.
\begin{thm}\label{T5}
Let $n\geq 3$ and $ m=x_1 < x_2 < \cdots < x_n=M$ be any   $n$ integers. Then
{ \begin{equation}\label{5t1}
		{m_4} \geq \begin{cases}
			\frac{(M-m)^4}{8n}+	\frac{(n-4)(n-3)(n-2)(3n^2-18n+20)}{240n},  &\text{if} ~~n \text{ is even},\\
			\frac{(M-m)^4}{8n}+	\frac{(n-3)(3n^4-36n^3+182n^2-444n+415)}{240n}, & \text{if}~~ n \text{ is odd}.
		\end{cases}
\end{equation}}
Furthermore, if $\bar{x}$ is an integer, then
{ \begin{equation}\label{5t2}
		{m_4} \geq \begin{cases}
			\frac{(M-m)^4}{8n}+\frac{(n-2)(3n^4-24n^3+92n^2-176n+120)}{240n},  &\text{if} ~~n \text{ is even},\\
			\frac{(M-m)^4}{8n}+	\frac{(n-1)(n-2)(n-3)(3n^2-12n+5)}{240n}, & \text{if}~~ n \text{ is odd}.
		\end{cases}
\end{equation}}
The equality in Theorem \ref{T5} holds precisely, when  $ x_2 , \cdots , x_{n-1}$ are $n-2$ consecutive integers, assuming that $n$ is odd, $\bar{x}$ is an integer, and $\bar{x}=\frac{m+M}{2}.$
\end{thm}
\begin{proof}
Let $\bar{x}\in \left[ x_k,x_{k+1}\right] $ with $1 \leq k \leq n-1.$ Then, we write {\small	\begin{equation}\label{5t3}
		\begin{aligned}
			\sum_{i=1}^{n}\left( x_i-\bar{x}\right) ^4&=	(M-\bar{x})^4+(\bar{x}-m)^4+\sum_{i=2}^{k} \left( x_i-\bar{x}\right) ^4+\sum_{i=k+1}^{n-1}\left( x_i-\bar{x}\right) ^4\\&\geq
			(M-\bar{x})^4+(\bar{x}-m)^4+\sum_{i=2}^{k} (k-i)^4+\sum_{i=k+1}^{n-1}(i-k-1)^4  ~~(\text{by} \eqref{4t3}~~ \text{and} ~~\eqref{4t4} )\\&\geq	\frac{(M-m)^4}{8}+\sum_{i=1}^{k-2} i^4+\sum_{i=1}^{n-k-2}i^4\\&=\frac{(M-m)^4}{8}+(n-3)\left( k^4-2nk^3+(2n^2-3n+4)k^2-(n^2-3n+4)k\right)\\&\quad+\frac{(n-3)}{30}\left( 6n^4-27n^3+49n^2-33n+20\right).
		\end{aligned}
\end{equation}}
It is evident that the right-hand side expression of \eqref{5t3} is monotonically decreasing on the interval $[1,\frac{n}{2}]$, and monotonically increasing on $[\frac{n}{2},n].$ However, since \( k \) is an integer, the minimum value of the expression is attained at  $k=\frac{n}{2}$, when $n$ is even  or $k=\frac{n\pm1}{2}$, when $n$ is odd. Hence, \eqref{5t1} follows.  
\\Let $\bar{x}$ be an integer. Then, \eqref{5t2} can be obtained by using \eqref{4t6} in a manner similar to how we prove \eqref{5t1}.
\end{proof}
We now demonstrate the effectiveness of our results for the case of $n$ distinct integers through an example.
\begin{ex}\label{se1}
Consider the set $T=\{1,2,3,4,5,6,8,9,10\}$. Then, from \eqref{2.25raj20}, we have $-20.3425 \leq m_3 \leq 25.4074,$ while our bound \eqref{c3} gives a better estimate: $ -7.1574 \leq m_3 \leq 12.2222. $ Similarly, from  \eqref{it3} and \eqref{i3}, we obtain \( 91.1250 \leq m_4 \leq 181.0022 \), while our bounds \eqref{3t11} and \eqref{5t1}, provide a better estimate: \( 103.9027 \leq m_4 \leq 162.5578 \).
\end{ex}
\section{Application to matrix theory}\label{sec3}
This section demonstrates the application of our results from Section 2 to matrix theory. For this, consider \( A \in \mathbb{M}_n \)  be an \( n \times n \) complex matrix with real eigenvalues arranged in descending order:
$
\lambda_1(A) \geq \lambda_2(A) \geq \cdots \geq \lambda_n(A)$. It is important to recall that the eigenvalues of an
$n\times n$ complex matrix are the roots of its characteristic polynomial, which are generally challenging to determine analytically. Bounds for eigenvalues in terms of easily measurable quantities like traces and norms of a matrix have been studied by several authors; see \cite{bha97,rkumar22,rkumar23,rkumar23a,raj15,ver24,wol80}. In the present context, Wolkowicz and Styan \cite{wol80} have obtained the following interval for the $j$-th $(1\leq j \leq n)$ eigenvalue of $A$,
{\scriptsize	\begin{equation}
		\begin{array}{lcl}
			\label{ieqi52}
			\frac{\text{tr}A}{n}-\sqrt{\dfrac{j-1}{n-j+1}}\left( 
			\frac{\text{tr}A^2}{n}-\left( \frac{\text{tr}A}{n}\right) ^{2}\right) ^{
				\frac{1}{2}}\leq
			\lambda _{j}(A)\leq \frac{\text{tr}A}{n}+\sqrt{\dfrac{n-j}{j}}\left( 
			\frac{\text{tr}A^2}{n}-\left( \frac{\text{tr}A}{n}\right) ^{2}\right) ^{
				\frac{1}{2}}. &  &  
		\end{array} 
\end{equation}}
In this section, we present a generalization of inequality \eqref{ieqi52} in Theorem \ref{mT2} under the condition that all eigenvalues of the matrix are real. Additionally, we extend inequality \eqref{mcc1} in the context of real functions in Theorem \ref{mT1}.
\begin{thm}\label{mT2}
	Let	$ A \in \mathbb{M}_{n} $ be any complex matrix with real eigenvalues  $\lambda_1(A)  \geq \lambda_2(A)\geq  \cdots \geq \lambda_{j}(A)\geq \cdots \geq \lambda_n(A)$ and let  $B=A-\frac{{\text{tr}A}}{n}I_n$. Then, for any positive integer $r $ and  for $1\leq j\leq n$,
	{\scriptsize 	\begin{equation}\label{mcc1}
			\frac{{\text{tr}A}}{n}-\left(\frac{(j-1)^{2r-1}{\text{tr}{B}^{2r}}}{(n-j+1)^{2r}+(n-j+1)(j-1)^{2r-1}}\right)      ^{\frac{1}{2r}}\leq\lambda_j(A)\leq \frac{{\text{tr}A}}{n}+\left(  \frac{(n-j)^{2r-1}{\text{tr}{B}^{2r}}}{j^{2r}+j(n-j)^{2r-1}} \right)     ^{\frac{1}{2r}}.
	\end{equation}}
	Furthermore, equality exists on the right-hand side if and only if 
	$$\lambda_1(A)=\lambda_2(A)=\cdots =\lambda_j(A)  ~ \text{and} ~\lambda_{j+1}(A)=\lambda_{j+2}(A)=\cdots=\lambda_n(A),$$
	and on the left-hand side if and only if 
	$$\lambda_1(A)=\lambda_2(A)=\cdots =\lambda_{j-1} (A)~ \text{and}~\lambda_{j}(A)=\lambda_{j+1}(A)=\cdots=\lambda_n(A).$$
\end{thm}
\begin{proof}
	The arithmetic mean of eigenvalues $\lambda_1(A), \lambda_2(A), \cdots , \lambda_n(A)$ is
	\begin{equation}\label{cc7}
		\bar{x}=\frac{1}{n}\sum_{i=1}^{n} \lambda_i(A) = \frac{1}{n}\text{tr}A.
	\end{equation} 
	Also, the eigenvalues of $B$ are $\lambda_{j}(A)-\frac{1}{n}\text{tr}A$. Therefore, we have
	\begin{equation}\label{mpe31}
		m_{2r}=\frac{1}{n}\sum_{j=1}^{n} \left(  \lambda_j(A)-\frac{1}{n}\text{tr}A\right) ^{2r}=\frac{1}{n}\text{tr}{B}^{2r}.
	\end{equation} 
	The inequality \eqref{mcc1} follows from \eqref{2e3} by substituting the values of $\bar{x}$ and $m_{2r}$ from \eqref{cc7} and \eqref{mpe31}, respectively.
\end{proof}
\begin{thm}\label{mT1}
	Let $ A={(a_{ij})} \in \mathbb{M}_{n} $ be any $n \times n$ complex matrix having real eigenvalues  $\lambda_1(A)  \geq \lambda_2(A)\geq \cdots \geq \lambda_{j}(A)\geq \cdots \geq \lambda_n(A)$. Let $\varphi : \mathbb{M}_{n} \longrightarrow \mathbb{R} $ be any real functional and  let $C=\left( A-\varphi\left( A\right) I_{n}\right) ^{q},~ B=C-\bar{c}I_n$, where $\bar{c}=\frac{{\text{tr}C}}{n}$ and $q$ is any positive odd integer.
	Then for any positive integer $r $ and $1\leq j \leq n$,
	\begin{equation}\label{mf1}
		\varphi\left( A\right)+ \left(\bar{c}- \alpha \right)^{\frac{1}{q}}  \leq\lambda_{j}(A) \leq \varphi\left( A\right)+\left(\bar{c}+\beta  \right)^{\frac{1}{q}},
	\end{equation} where $\alpha=\left(\frac{(j-1)^{2r-1}{\text{tr}{B}^{2r}}}{(n-j+1)^{2r}+(n-j+1)(j-1)^{2r-1}}\right) ^{\frac{1}{2r}} \quad\text{and}\quad\beta =\left(  \frac{(n-j)^{2r-1}{\text{tr}{B}^{2r}}}{j^{2r}+j(n-j)^{2r-1}} \right)^{\frac{1}{2r}}$.\\
	Furthermore, equality exists on the right-hand side of \eqref{mf1} if and only if 
	$$\lambda_1(A)=\lambda_2(A)=\cdots =\lambda_j(A)  ~ \text{and} ~\lambda_{j+1}(A)=\lambda_{j+2}(A)=\cdots=\lambda_n(A),$$
	and on the left-hand side if and only if 
	$$\lambda_1(A)=\lambda_2(A)=\cdots =\lambda_{j-1} (A)~ \text{and}~\lambda_{j}(A)=\lambda_{j+1}(A)=\cdots=\lambda_n(A).$$
\end{thm}
\begin{proof}
	The eigenvalues of $C$ are  $\left(  \lambda_i(A)-\varphi\left( A\right)\right)^{q}=c_i$ (say). Hence, the proof of the theorem follows by applying Theorem \ref{mT2}. 
\end{proof}

\begin{ex} \label{x4}
	Let
	\begin{equation*}
		A_{1}=\begin{bmatrix}
			4&0&2&3\\
			0&5&0&1\\
			2&0&6&3\\
			3&1&0&7\\
		\end{bmatrix}.
	\end{equation*}
	From  Wolkowicz and Styan bound \eqref{ieqi52}, we have $5.500\leq \lambda_{1}(A_{1})\leq 10.4749$, while our bound  \eqref{mf1} provides a better estimate:  $8.1261 \leq \lambda_{1}(A_{1}) \leq 9.3775 $ for $\varphi\left( A_{1}\right)=\frac{a_{11}+a_{22}}{2}, ~q=5, ~r=1$. Likewise, from \eqref{ieqi52} $\lambda_{3}(A_{1}) \leq 7.1583$, while \eqref{mcc1} for $r=3$ gives a better approximation $\lambda_{3}(A_{1}) \leq 6.9886 $.\\Furthermore, from \eqref{ieqi52}, we have $\lambda_{2}(A_{1}) \geq 3.8417$ while \eqref{mf1} for $\varphi\left( A_{1}\right)=a_{41},~q=3,~r=1$ gives a better estimate: $\lambda_{2}(A_{1}) \geq5.3900$. Similarly, using \eqref{ieqi52} we get $\lambda_{4}(A_{1}) \geq 0.5250$, while \eqref{mf1} for $\varphi\left( A_{1}\right)=\frac{{\text{tr}A_{1}}}{n},~q=3,~r=2$ provides a better estimate: $\lambda_{4}(A_{1}) \geq 1.2534$. This example shows that our Theorems \ref{mT2}?\ref{mT1} provide a better estimate than Theorem 2.2 of Wolkowicz and Styan \cite{wol80}.  
\end{ex}
Matrices with integer eigenvalues arise naturally in various areas of mathematics and applications, including number theory, combinatorics, and control theory. The spread of a matrix $A \in\mathbb{M}_{n}$ is denoted as 
$\text{spd}\left( A\right)$, is defined as the maximum difference between any two eigenvalues of $A$, that is,
\begin{equation*}\label{isd1}
	\text{spd}\left( A\right)  =\underset{j,i}\max{|\lambda _{j }(A)-\lambda _{i }(A)|}, ~~1\leq i\leq n, 1\leq j\leq n.
\end{equation*}
Bounds for the spread of a matrix have been investigated by various authors; refer to \cite{bha97,mir56,nagy18,raj20} and references therein. 
\begin{thm}\label{mT3}
	Let $ A\in \mathbb{M}_{n} $ be any $n \times n$ complex matrix having    $n$ distinct integer eigenvalues and let  $B=A-\frac{{\text{tr}A}}{n}I_n$. Then for $n \geq 3,$
	{\small \begin{equation*}\label{m4t3}
			\hspace{.3cm}			\text{spd}(A)		 	\leq  \begin{cases}
				\left(	{16\text{tr}{B}^{4}}-\frac{(n-4)(n-3)(n-2)(3n^2-18n+20)}{240} \right) ^\frac{1}{4},  &\text{if} ~~n \text{ is even},\\
				\left(	{16\text{tr}{B}^{4}}-	\frac{(n-3)(3n^4-36n^3+182n^2-444n+415)}{240} \right) ^\frac{1}{4}, & \text{if}~~ n \text{ is odd}.
			\end{cases}
	\end{equation*}}Furthermore, if $n$ divides $\text{tr}A$, then
	{\small \begin{equation*}\label{m4t2}
			\hspace{.3cm}			\text{spd}(A)		 	\leq  \begin{cases}
				\left(	{16\text{tr}{B}^{4}}-\frac{(n-2)(3n^4-24n^3+92n^2-176n+120)}{240} \right) ^\frac{1}{4},  &\text{if} ~~n \text{ is even},\\
				\left(	{16\text{tr}{B}^{4}}-	\frac{(n-1)(n-2)(n-3)(3n^2-12n+5)}{240} \right) ^\frac{1}{4}, & \text{if}~~ n \text{ is odd}.
			\end{cases}
	\end{equation*}}	
\end{thm}
\begin{proof}
	The proof directly follows from Theorem \ref{T5}.
\end{proof}
\begin{thm}\label{mT4}
	Let $ A\in \mathbb{M}_{n} $ be any $n \times n$ complex matrix having    $n$ distinct integer eigenvalues and let  $B=A-\frac{{\text{tr}A}}{n}I_n$. Then for $n \geq 3,$
	{\small\begin{equation*}\label{m4t4}
			\text{spd}(A)	 \geq4\sqrt{2} \left(	\frac{\text{tr}{B}^{4}}{n}-\left( \frac{\text{tr}{B}^{2}}{n}\right) ^3-\frac{{\text{tr}{B}^{3}}^2}{n\text{tr}{B}^{2}} +\beta_3 \right) ^\frac{1}{4},
	\end{equation*}}
	where 		{\scriptsize \begin{equation*}\label{m3t2}
			\beta_3 =\begin{cases}
				\frac{1}{480}(n-2)(n-4)(4n^2-11n+2), &\text{if} ~~~~n \text{ is even},\\
				
				\frac{1}{480n}(n-1)(n-3)(4n^3-19n+32n-5), &\text{if} ~~~~n \text{ is odd}.
			\end{cases}
	\end{equation*}}
\end{thm}\begin{proof}
	The proof directly follows from Corollary \ref{C6}.
\end{proof}
\begin{ex} \label{x8}
	{	Let
		\begin{equation*}
			A=\begin{bmatrix}
				2 & -1 & -2 & 0 & 1 \\
				2 & 2 & 1 & 1 & 2 \\
				-2 & 1 & 1 & 1 & 1 \\
				-1 & 1 & 2 & 1 & -1 \\
				2 & 0 & 2 & 0 & 2
			\end{bmatrix}.
	\end{equation*}}The eigenvalues of $A$ are $\{-2,1,2,3,4\}.$ From the Nagy inequality \eqref{i2}, $	\text{spd}(A)	 	\leq 6.5115$ and from Sharma et al. inequality \eqref{i3} for $r=2$,  $\text{spd}(A) 	\leq 6.3648$, while our Theorem \ref{mT3} gives a better approximation:   $\text{spd}(A)	 	\leq 6.3638$.\\
	Likewise, from Sharma et al. inequality \eqref{it4} $\text{spd}(A)	 \geq 5.5119 $, while our Theorem \ref{mT4} provides a better estimate:  $\text{spd}(A)	 \geq 5.5128$.
\end{ex}
\section{Application to the theory of polynomial equations}\label{sec4}
In mathematics, the determination of exact roots of polynomial equations has remained a longstanding and fundamental challenge. While explicit solutions exist for polynomials of degree four or lower, however no general method is available for solving higher-degree polynomials exactly. Consequently, researchers have focused on deriving approximations for the roots in terms of the polynomial's coefficients; see \cite{mard02,rah02,raj18,raj20,ver24}.
Consider,
\begin{equation}\label{ipe1}
	f\left( x \right) =x^n+b_2x^{n-2}+b_4x^{n-3}+\cdots+b_n=0,
\end{equation} is a real polynomial equation with real roots $ x_1\geq  x_2 \geq \dots \geq x_n .$ Let $\bar{x}$ represents the arithmetic mean of $x_{j}$'s. Then
from \cite{raj18}, we have
\begin{equation}\label{ipe2}
	\begin{split}
		m_1&=\frac{1}{n}\sum_{j=1}^{n}{x_{j}}=0=\bar{x},~m_2=\frac{1}{n}\sum_{j=1}^{n}{x_{j}}^2=-\frac{2}{n}b_2,\\m_3&=\frac{1}{n}\sum_{j=1}^{n}{x_{j}}^3=-\frac{3}{n}b_3,
		m_4=\frac{1}{n}\sum_{j=1}^{n}{x_{j}}^4=\frac{2}{n}({b_2}^2-2b_4).
	\end{split}
\end{equation}
From \cite{wol80}, we have 
\begin{equation}\label{ipe3}
	-\sqrt{\frac{-2(j-1)b_2}{n-j+1}}\leq x_{j} \leq \sqrt{\frac{-2(n-j)b_2}{j}}, \quad 1 \leq j \leq n.
\end{equation}
In a similar spirit, we provide an interval for the root $x_j$ of the polynomial equation \eqref{ipe1} in terms of its coefficients using central moments.
\begin{thm}\label{pT1}
	Let $n\geq5$ and let $x_1 \geq x_2 \geq  \cdots \geq x_{j}\geq \cdots \geq x_n$ be  roots of \eqref{ipe1}. For $1\leq j \leq n$,
	\begin{equation}\label{pe2}
		-	\left(\frac{2(j-1)^{3}({b_2}^2-2b_4)}{(n-j+1)^{4}+(n-j+1)(j-1)^{3}} \right)      ^{\frac{1}{4}}\leq	x_j \leq \left(  \frac{2(n-j)^{3}({b_2}^2-b_4)}{j^{4}+j(n-j)^{3}} \right)     ^{\frac{1}{4}}.
	\end{equation}
\end{thm}
\begin{proof}
	To prove \eqref{pe2}, we apply Theorem \ref{L1}. By substituting $r=2$ in \eqref{2e3}, we get 
	\begin{equation}\label{ppe1} 
		\bar{x}-\left(\frac{n(j-1)^{3}m_{4}}{(n-j+1)^{4}+(n-j+1)(j-1)^{3}} \right)      ^{\frac{1}{4}}\leq	x_j \leq \bar{x}+\left(  \frac{n(n-j)^{3}m_{4}}{j^{4}+j(n-j)^{3}} \right)     ^{\frac{1}{4}}.
	\end{equation}
	Inserting the values of $\bar{x}$ and $m_{4}$ from \eqref{ipe2} in \eqref{ppe1}, we get \eqref{pe2}. 
\end{proof}\begin{thm}\label{pT4}
	Let $f(x)$ be a polynomial as defined in \eqref{ipe1}, with $n$ distinct integer roots. Then for $n\geq 3$,
	{ \begin{equation*}\label{p4t2}
			{b_2}^2-2b_4 \geq \begin{cases}
				\frac{n(3n^4+20n^2-8)}{480},  &\text{if} ~~n \text{ is even},\\
				\frac{n(n+1)(n-1)(3n^2-7)}{480}, & \text{if}~~ n \text{ is odd}.
			\end{cases}
	\end{equation*}}
\end{thm}
\begin{proof}
	The proof follows from Theorem \ref{T4} by using \eqref{ipe2}.
\end{proof}
The term span of a polynomial equation refers to $\text{spn}\left( p\right) =\max_{j,i}|x_j-x_i|.$ Several researchers have derived inequalities concerning the span of polynomials; see \cite{rah02,raj10,raj18,raj20,ver24} and the references therein.

\begin{thm}\label{P3}
	Let $f(x)$ be a polynomial as defined in \eqref{ipe1}, with $n$ distinct integer roots. Then for $n\geq 3$,
	{ \begin{equation*}\label{p4t2}
			\hspace{.3cm}				\text{spn}\left( f\right)	 	\leq  \begin{cases}
				\left(	{16{b_2}^2-32b_4}-\frac{(n-2)(3n^4-24n^3+92n^2-176n+120)}{240} \right) ^\frac{1}{4},  &\text{if} ~~n \text{ is even},\\
				\left(	{16{b_2}^2-32b_4}-	\frac{(n-1)(n-2)(n-3)(3n^2-12n+5)}{240} \right) ^\frac{1}{4}, & \text{if}~~ n \text{ is odd}
			\end{cases}
	\end{equation*}}and
	{ \begin{equation*}\label{p4t3}
			\hspace{.3cm}				\text{spn}\left( f\right)	 	\geq4\sqrt{2} \left(	\frac{2({b_2}^2-2b_4)}{n}+ \frac{8{b_2}^3} {n^3}+\frac{9b_3^2}{2nb_2} +\beta_4 \right) ^\frac{1}{4}, 
	\end{equation*}}where 
	{ \begin{equation*}\label{m3t2}
			\beta_4 =\begin{cases}
				\frac{1}{480}(n-2)(n-4)(4n^2-11n+2), &\text{if} ~~~~n \text{ is even},\\
				
				\frac{1}{480n}(n-1)(n-3)(4n^3-19n+32n-5), &\text{if} ~~~~n \text{ is odd}.
			\end{cases}
	\end{equation*}}
\end{thm}
\begin{proof}
	The proof follows from Theorem \ref{T5}, Corollary \ref{C6} and using relations \eqref{ipe2}.
\end{proof}
\noindent We now illustrate an example to show the effectiveness of our Theorems.
\begin{ex} \label{x3}
	Let 
	\begin{equation*}
		f\left( x \right) =x^5-53x^3-24x^2+412x-336=0.
	\end{equation*}
	The actual roots of $f(x)$ are $ \left\lbrace  7,2,1,-4,-6 \right\rbrace.$
	\begin{center}
		\begin{tabular}{|c|c|c|}  
			\hline
			Roots&From \eqref{ipe3}  &From \eqref{pe2} for $r=2$\\ \hline
			$x_{1} \leq$ & 9.2087&7.9070 \\ \hline  
			$x_{2} \geq $&-2.3022 & -1.9768\\ \hline  
			$x_4 \leq$& 2.3022&1.9768\\ \hline 
			$x_{5}\geq $&-9.2087&-7.9070\\ \hline
		\end{tabular}  
	\end{center}From the Nagy bound \eqref{i2}, $	\text{spn}\left( f\right)	 	\leq 14.5602$ and from the bound of Sharma et al. \eqref{i3} for $r=2$,  $\text{spn}\left( f\right)\leq 13.3496$, while our Theorem \ref{P3} provides a better approximation:   $\text{spn}\left( f\right)\leq 13.3494$. Likewise, from the bound of Sharma et al. \eqref{it4}, $\text{spn}\left( f\right)	\geq 12.0986 $, while our Theorem \ref{P3} gives a better estimate: $\text{spn}\left(   f\right)	 \geq 12.0987$.
\end{ex}

\subsection*{Acknowledgments}
The research of first author is supported by the University Grants Commission (UGC), Government of India. The second author is supported by the National Board for Higher Mathematics (NBHM), Department of Atomic Energy (DAE), India (No. 02011/30/2025/NBHM(R.P.)/R\&D-II/9676).
\section*{Declarations}
\subsection*{Funding}
The authors received no funding for the preparation of this paper beside grant mentioned above.
\subsection*{Conflict of interest} The authors have no conflicts of interest associated with this publication.
\subsection*{Data availability} No data associated with this publication.

\end{document}